\documentclass[12pt]{amsart}
\usepackage{amsmath,latexsym,amsfonts,amssymb,amsthm}
\usepackage{geometry}
\usepackage{hyperref}
\usepackage{mathrsfs}
\usepackage{graphicx,color}
\usepackage{tikz-cd}
\usepackage[all,cmtip]{xy}
\usepackage{enumitem}
\usepackage{bbm}

\newtheorem{prop}{Proposition}
\newtheorem{thm}[prop]{Theorem}
\newtheorem{cor}[prop]{Corollary}
\newtheorem{conj}[prop]{Conjecture}
\newtheorem{lem}[prop]{Lemma}

\theoremstyle{definition}
\newtheorem{que}[prop]{Question}
\newtheorem{defn}[prop]{Definition}

\newtheorem{rem}[prop]{\it Remark}

\numberwithin{equation}{section}

\newcommand{\bC}{\mathbb{C}}
\newcommand{\bR}{\mathbb{R}}
\newcommand{\bA}{\mathbb{A}}
\newcommand{\bQ}{\mathbb{Q}}
\newcommand{\bZ}{\mathbb{Z}}

\newcommand{\bN}{\mathbb{N}}
\newcommand{\bk}{\mathbbm{k}}
\newcommand{\bK}{\mathbb{K}}

\newcommand{\cX}{\mathcal{X}}
\newcommand{\cY}{\mathcal{Y}}

\newcommand{\cO}{\mathcal{O}}
\newcommand{\cL}{\mathcal{L}}

\newcommand{\cE}{\mathcal{E}}

\newcommand{\cD}{\mathcal{D}}
\newcommand{\cH}{\mathcal{H}}

\newcommand{\cS}{\mathcal{S}}

\newcommand{\cZ}{\mathcal{Z}}

\newcommand{\fa}{\mathfrak{a}}
\newcommand{\fb}{\mathfrak{b}}
\newcommand{\fc}{\mathfrak{c}}
\newcommand{\fm}{\mathfrak{m}}

\newcommand{\Spec}{\mathrm{Spec}}
\newcommand{\Supp}{\mathrm{Supp}}
\newcommand{\lct}{\mathrm{lct}}

\newcommand{\hvol}{\widehat{\mathrm{vol}}}

\newcommand{\ord}{\mathrm{ord}}

\newcommand{\Val}{\mathrm{Val}}

\newcommand{\e}{\mathrm{e}}

\newcommand{\vol}{\mathrm{vol}}
\newcommand{\depth}{\mathrm{depth}}
\newcommand{\rom}[1]{\lowercase\expandafter{\romannumeral #1\relax}}

\begin{document}

\title[semi-continuity problem of normalized
volumes of singularities]
{On the semi-continuity problem of
normalized volumes of singularities}
\author{Yuchen Liu}
\address{Department of Mathematics, Yale University, New Haven, CT 06511, USA.}
\email{yuchen.liu@yale.edu}

\date{\today}

\begin{abstract}
We show that in any $\bQ$-Gorenstein flat family
of klt singularities, normalized volumes
can only jump down at countably many subvarieties.
A quick consequence is that smooth points 
have the largest normalized volume among 
all klt singularities. Using an alternative
characterization of K-semistability developed
by Li, Xu and the author, we show that K-semistability
is a very generic or empty condition in any $\bQ$-Gorenstein
flat family of log Fano pairs.
\end{abstract}

\maketitle

\section{Introduction}

Given an $n$-dimensional complex klt singularity $(x\in (X,D))$, Chi Li \cite{li15a} 
introduced the \emph{normalized volume} function on the space $\Val_{x,X}$ of real valuations of $\bC(X)$ centered at $x$. 
More precisely, for any such valuation $v$, its normalized volume
is defined as $\hvol_{x,(X,D)}(v):=A_{(X,D)}(v)^n\vol(v)$, where $A_{(X,D)}(v)$
is the log discrepancy of $v$ with respect to $(X,D)$ according to \cite{jm12, bdffu15}, and 
$\vol(v)$ is the volume of $v$ according to \cite{els03}. 
Then we can define the \emph{normalized volume of a klt singularity}
$(x\in (X,D))$ by 
\[
 \hvol(x,X,D):=\min_{v\in\Val_{x,X}}\hvol_{x,(X,D)}(v)
\]
where the existence of minimizer of $\hvol$ was
shown recently by Blum \cite{blu16}. We also denote
$\hvol(x,X):=\hvol(x,X,0)$.

The normalized volume of a klt singularity
$x\in (X,D)$ carries some interesting information of its geometry 
and topology. It was shown by Xu and the author
that $\hvol(x,X,D)\leq n^n$ and equality holds
if and only if $(x\in X\setminus \Supp(D))$ is
smooth (see \cite[Theorem A.4]{liux17} or Theorem \ref{maxhvol}).
By \cite{xu14} the local algebraic fundamental group $\hat{\pi}_1^{\mathrm{loc}}(X,x)$
of a klt singularity $x\in X$ is always finite.
Moreover, assuming the conjectural finite degree
formula of normalized volumes \cite[Conjecture 4.1]{liux17}, then the size of $\hat{\pi}_1^{\mathrm{loc}}(X,x)$ is bounded from above by 
$n^n/\hvol(x,X)$ (see Remark \ref{r_localpi1}).
If $X$ is a Gromov-Hausdorff
limit of K\"ahler-Einstein Fano 
manifolds, then Li and Xu \cite{lx17} showed that $\hvol(x,X)=n^n\cdot\Theta(x,X)$ 
where $\Theta(x,X)$ is the volume density 
of a closed point $x\in X$ (see \cite{hs16, ss17} for 
background materials).
\medskip

In this article, we investigate the behavior of normalized volumes
of singularities under deformation. We first state the following
natural conjecture on constructibility and lower 
semi-continuity of normalized volumes of klt 
singularities (see also \cite[Conjecture 4.11]{xu17}).

\begin{conj}\label{mainconj}
Let $\pi:(\cX,\cD)\to T$ together with a section $\sigma: T\to \cX$ be a
 $\bQ$-Gorenstein flat family of complex klt singularities over a normal
 variety $T$. Then the function $t\mapsto\hvol(\sigma(t),\cX_{t},\cD_t)$
 on $T(\bC)$ is constructible and lower semi-continuous with respect to the Zariski topology.
\end{conj}

Our first main result partially confirms Conjecture 
\ref{mainconj} by showing that normalized 
volumes of klt singularities satisfy weak 
lower semi-continuity in the sense that
they can only jump down at countably
many subvarieties.

\begin{thm}\label{mainthm}
Let $\pi:(\cX,\cD)\to T$ together with a section $\sigma: T\to \cX$ be a
 $\bQ$-Gorenstein flat family of complex klt singularities.
Then for any closed point $o\in T$, there exists an intersection
$U$ of countably many Zariski open neighborhoods of $o$, 
such that $\hvol(\sigma(t),\cX_t,\cD_t)\geq\hvol(\sigma(o),\cX_o,\cD_o)$
for any closed point $t\in U$.
\end{thm}

A quick consequence of Theorem \ref{mainthm} is that smooth points
have the largest normalized volumes among all klt singularities
(see Theorem \ref{maxhvol} or \cite[Theorem A.4]{liux17}).
\medskip

To verify the Zariski openness of K-semistability is 
an important step in the construction of algebraic moduli space
of K-polystable $\bQ$-Fano varieties. In a smooth family of Fano manifolds,
Odaka \cite{oda13} and Donaldson \cite{don15} showed that the locus of fibers admitting K\"ahler-Einstein metrics (or equivalently, being K-polystable) with
discrete automorphism groups is Zariski open.  This was
generalized by Li, Wang and Xu \cite{lwx14} where they proved the
Zariski openness of K-semistability in a $\bQ$-Gorenstein
flat families of smoothable $\bQ$-Fano varieties
in their construction of the proper moduli space of smoothable
K-polystable $\bQ$-Fano varieties (see \cite{ssy16, oda15} for
related results).
A common feature is that analytic methods were used essentially in proving these results.

Using the alternative characterization of K-semistability
by the affine cone construction 
developed by Li, Xu and the author in \cite{li15b, ll16, lx16}, 
we prove the following result on weak openness
of K-semistability as an application of Theorem \ref{mainthm}. Unlike the results described in the previous paragraph, our result is proved using purely algebraic method and hence can be applied to $\bQ$-Fano families with non-smoothable fibers (or more generally, families of log Fano pairs).

\begin{thm}\label{openk}
Let $\varphi:(\cY,\cE)\to T$ be a $\bQ$-Gorenstein flat 
family of complex log Fano pairs over a normal base $T$.
Assume that $(\cY_o,\cE_o)$ is log K-semistable 
for some closed point $o\in T$. Then
\begin{enumerate}
 \item There exists an intersection $U$ of countably many Zariski
 open neighborhoods of $o$, such that $(\cY_t,\cE_t)$ is
 log K-semistable for any closed point $t\in T$.
 In particular, $(\cY_t,\cE_t)$ is log K-semistable for a very general 
 closed point $t\in T$.
 \item Denote by $\eta$ the generic point of $T$, then
 the geometric generic fiber $(\cY_{\bar{\eta}},\cE_{\bar{\eta}})$
 is log K-semistable.
 \item Assume Conjecture \ref{mainconj} is true, then such $U$
 from (1) can be chosen as a genuine Zariski open neighborhood of $o$.
\end{enumerate}
\end{thm}

The following corollary follows easily from Theorem \ref{openk}.
\begin{cor}\label{specialdeg}
 Suppose a complex log Fano pair $(Y,E)$ specially degenerates to a log K-semistable
 log Fano pair $(Y_0,E_0)$, then $(Y,E)$ is also log K-semistable.
\end{cor}

Our strategy to prove Theorem \ref{mainthm} is to study various invariants
of ideals instead of valuations. From the author's characterization
of normalized volume by normalized multiplicities of ideals (see
\cite[Theorem 27]{liu16} or Theorem \ref{eqhvol}), we know
that $\hvol(\sigma(t),\cX_t,\cD_t)=\inf_{\fa}\lct(\cX_t,\cD_t;\fa)^n\cdot\e(\fa)$
where the infimum is taken over all ideals $\fa$ cosupported at
$\sigma(t)$. These ideals are parametrized by some relative 
Hilbert scheme of $\cX/T$ with countably many components. Clearly $\fa\mapsto\lct(\cX_t,\cD_t;\fa)$ is lower semi-continuous
on the Hilbert scheme,  but $\fa\mapsto\e(\fa)$ is usually upper semi-continuous
which makes the desired lower semi-continuity of $\lct(\fa)^n\cdot\e(\fa)$ obscure.
To fix this issue, we introduce the \emph{normalized colength of singularities}
by taking the infimum of $\lct(\cX_t,\cD_t;\fa)^n\cdot\ell(\cO_{\sigma(t),\cX_t}/\fa)$
for certain classes of ideals $\fa$. The normalized colength function
behaves better in families since the colength function $\fa\mapsto\ell(\cO_{\sigma(t),\cX_t}/\fa)$
is always locally constant in the Hilbert scheme, so  $\fa\mapsto\lct(\cX_t,\cD_t;\fa)^n\cdot\ell(\cO_{\sigma(t),\cX_t}/\fa)$
is lower semi-continuous on the Hilbert scheme. Then we prove
a key equality between the asymptotic normalized colength 
and the normalized volume (see Theorem \ref{hvolcolength}) using
local Newton-Okounkov bodies following \cite{cut13, kk14} (see
Lemma \ref{colengthmult}) and 
convex geometry (see Appendix \ref{app}). Putting these ingredients
together, we get the proof of Theorem \ref{mainthm}.
\medskip

This paper is organized as follows. In Section \ref{prelim}, we give
the preliminaries including notations, normalized volumes of singularities, and $\bQ$-Gorenstein flat
families of klt pairs. In Section \ref{hatl},
we introduce the concept of normalized colengths of singularities.
We show in Theorem \ref{hvolcolength} that the normalized volume
of a klt singularity is the same as its asymptotic normalized colength.
The proof of Theorem \ref{hvolcolength} uses a comparison
of colengths and multiplicities established in Lemma \ref{colengthmult}. In Section \ref{sec_fieldext},
we study the normalized volumes and normalized colength after
algebraically closed field extensions. The proofs of main theorems
are presented in Section \ref{sec_proofs}. We give applications
of our main theorems in Section \ref{sec_appl}. Theorem \ref{maxhvol}
generalizes the inequality part of \cite[Theorem A.4]{liux17}.
We give an effective upper bound on the degree of finite quasi-\'etale maps 
over klt singularities on Gromov-Hausdorff limits of K\"ahler-Einstein
Fano manifolds (see Theorem \ref{ghlimit}). In Appendix \ref{app}
we provide certain convex geometric results on lattice points counting that are needed
in proving Lemma \ref{colengthmult}.

\subsection*{Acknowledgements} I would like to thank 
 Harold Blum, Chi Li and Chenyang Xu for fruitful discussions. I wish to
thank J\'anos Koll\'ar, Linquan Ma, Mircea Musta\c{t}\u{a}, Sam Payne,
Xiaowei Wang, and Ziquan Zhuang for helpful comments. I am also grateful to Ruixiang Zhang for his help
on Proposition \ref{convexgeo}.

\section{Preliminaries}\label{prelim}
\subsection{Notations}
In the rest of this paper, all varieties are assumed to be irreducible, reduced, and defined over a (not necessarily 
algebraically closed) field $\bk$ of characteristic $0$.
For a variety $T$ over $\bk$, we denote the residue field of
any scheme-theoretic point $t\in T$ by $\kappa(t)$.
Let $\pi:\cX\to T$ between varieties over $\bk$, we denote by $\cX_t:=\cX\times_{T}\Spec(\kappa(t))$ the scheme theoretic fiber over $t\in T$. We also denote
the geometric fiber of $\pi$ over $t\in T$ by $\cX_{\bar{t}}:=\cX\times_{T}\Spec(\overline{\kappa(t)})$.
Suppose $X$ is a variety over $\bk$ and $x\in X$ is a $\bk$-rational point. Then for any
field extension $\bK/\bk$, we denote $(x_{\bK}, X_{\bK}):=(x,X)\times_{\Spec(\bk)}\Spec(\bK)$.

Let $X$ be a normal variety over $\bk$. Let $D$ be an effective $\bQ$-divisor on $X$. We say that $(X,D)$ is a \emph{Kawamata log terminal (klt) pair} if $(K_X+D)$ is $\bQ$-Cartier and  $\ord_E(K_Y-f^*(K_X+D))>-1$ (equivalently, $A_{(X,D)}(\ord_E)>0$) for any prime divisor $E$ on some log resolution $f:Y\to (X,D)$. A klt pair $(X,D)$ is called \emph{a log Fano pair} if in addition $X$ is proper and $-(K_X+D)$ is ample. A klt pair $(X,D)$ together with a closed point $x\in X$ is called a \emph{klt singularity} $(x\in (X,D))$.

Let $(X,D)$ be a klt pair. For an ideal sheaf $\fa$ on $X$, we define
the \emph{log canonical threshold of $\fa$ with respect to $(X,D)$} by 
\[
\lct(X,D;\fa):=\inf_E\frac{1+\ord_E(K_Y-f^*(K_X+D))}{\ord_E(\fa)},
\]
where the infimum is taken over all prime divisors $E$ on a log resolution $f:Y\to (X,D)$. We will also use the notation $\lct(\fa)$ as abbreviation of $\lct(X,D;\fa)$ once the klt pair $(X,D)$ is specified.
If $\fa$ is co-supported at a single closed point $x\in X$, we define the \emph{Hilbert-Samuel multiplicity} of $\fa$ as
\[
\e(\fa):=\lim_{m\to\infty}\frac{\ell(\cO_{x,X}/\fa^m)}{m^n/n!}
\]
where $n:=\dim(X)$ and $\ell(\cdot)$ is the length of an Artinian ring.

\subsection{Normalized volumes of singularities}
 Let $\bk$ be an algebraically closed field of characteristic $0$. For
 an $n$-dimensional klt singularity $x\in (X,D)$ over $\bk$, C. Li \cite{li15a} introduced the normalized volume function $\hvol_{x,(X,D)}:\Val_{X,x}\to \bR_{>0}\cup\{+\infty\}$ where $\Val_{X,x}$ is the space of all real valuations of $\bk(X)$ centered at $x$. Then we can define the normalized volume of the singularity
 $x\in (X,D)$ to be
 \[
 \hvol(x,X,D):=\inf_{v\in\Val_{X,x}}\hvol_{x,(X,D)}(v).
 \]
 If in addition $\bk$ is uncountable, then Blum \cite{blu16}
 proved the existence of a $\hvol$-minimizng valuation. The following characterization of normalized volumes using log canonical thresholds and multiplicities of ideals is crucial in our study. Note that the right hand side of \eqref{lcte} was studied by de Fernex, Ein and Musta\c{t}\u{a} \cite{dfem04}
 when $x\in X$ is smooth and $D=0$.
 
 \begin{thm}[{\cite[Theorem 27]{liu16}}]\label{eqhvol}
 With the above notation, we have
 \begin{equation}\label{lcte}
 \hvol(x, X, D)=\inf_{\fa\colon \fm_{x}\textrm{-primary}}\lct(X,D;\fa)^n\cdot\e(\fa).
 \end{equation}
 
 \end{thm}
 
 The following theorem provides an alternative characterization
 of K-semistability using the affine cone construction. Here we
 state the most general form, and special cases can be found
 in \cite{li15b, ll16}.
 
\begin{thm}[{\cite[Proposition 4.6]{lx16}}]\label{ksscone}
Let $(Y,E)$ be a log Fano pair of dimension $(n-1)$ over an algebraically closed field $\bk$ of characteristic $0$. For $r\in\bN$ satisfying
$L:=-r(K_Y+E)$ is Cartier, the affine cone $X=C(Y,L)$ is
defined by $X:=\Spec\oplus_{m\geq 0}H^0(Y,L^{\otimes m})$. Let $D$ be the $\bQ$-divisor on $X$ corresponding to $E$. Denote by $x$ the cone vertex of $X$. Then 
\[
\hvol(x,X,D)\leq r^{-1}(-K_Y-E)^{n-1},
\]
and the equality holds if and only if $(Y,E)$ is log K-semistable.
\end{thm}

\subsection{$\bQ$-Gorenstein flat families of klt pairs}

In this section, the field $\bk$ is not assumed to be algebraically closed.
\begin{defn}\begin{enumerate}[label=(\alph*)]
\item Given a normal variety $T$, a \emph{$\bQ$-Gorenstein flat family of klt pairs over} $T$ consists of
a surjective flat morphism $\pi:\cX\to T$ from a variety $\cX$, and 
an effective $\bQ$-divisor $\cD$ on $\cX$ avoiding codimension $1$ singular points of $\cX$,
such that the following conditions hold:
\begin{itemize}
    \item All fibers $\cX_t$ are connected, normal and not contained in $\Supp(\cD)$;
    \item $K_{\cX/T}+\cD$ is $\bQ$-Cartier;
    \item $(\cX_t,\cD_t)$ is a klt pair for any $t\in T$.
\end{itemize}
\item A $\bQ$-Gorenstein flat family of klt pairs $\pi:(\cX,\cD)\to T$ together with a section $\sigma:T\to\cX$ is called a \emph{$\bQ$-Gorenstein flat family of klt singularities}. We denote
by $\sigma(\bar{t})$ the unique closed point of $\cX_{\bar{t}}$ lying over $\sigma(t)\in\cX_t$.
\end{enumerate}
\end{defn}

\begin{prop} Let $\pi:(\cX,\cD)\to T$ be a $\bQ$-Gorenstein flat family of klt pairs over a normal variety $T$. Then
\begin{enumerate}
    \item There exists a closed subset $\cZ$ of $\cX$ 
    of codimension at least $2$, such that $\cZ_t$ has codimension at least $2$ in $\cX_t$ for every $t\in T$, and $\pi:\cX\setminus \cZ\to T$ is a smooth morphism.
    \item $\cX$ is normal.    
    \item For any morphism $f:T'\to T$ from a normal variety $T'$
    to $T$, the base change $\pi_{T'}:(\cX_{T'},\cD_{T'})=(\cX,\cD)
    \times_{T}T'\to T'$ is a $\bQ$-Gorenstein flat family of klt 
    pairs over $T'$, and $K_{\cX_{T'}/T'}+\cD_{T'}=g^*(K_{\cX/T}+\cD)$
    where $g:\cX_{T'}\to \cX$ is the base change of $f$.
\end{enumerate}
\end{prop}

\begin{proof}
(1) Assume $\pi$ is of relative dimension $n$. Let 
$\cZ:=\{x\in \cX\mid \dim_{\kappa(x)}\Omega_{\cX/T}\otimes\kappa(x)>n\}$.
It is clear that $\cZ$ is Zariski closed. Since $\bk$ is of characteristic $0$,
$\cZ_t=\cZ\cap\cX_t$ is the singular locus of $\cX_t$. Hence $\mathrm{codim}_{\cX_t}\cZ_t\geq 2$
 because $\cX_t$ is normal.

(2) From (1) we know that $\cZ$ is of codimension at least $2$ in 
$\cX$, and $\cX\setminus \cZ$  is smooth over $T$. Thus
$\cX\setminus (\cZ\cup\pi^{-1}(T_{\mathrm{sing}}))$ is regular,
and $\cZ\cup\pi^{-1}(T_{\mathrm{sing}})$ has codimension at least 
$2$ in $\cX$. So $\cX$ satisfies property $(R_1)$. Since $\pi$ is flat,
 for any point $x\in\cX_t$ we have $\depth(\cO_{x,\cX})=\depth(\cO_{x,\cX_t})+\depth(\cO_{t,T})$
 by \cite[(21.C) Corollary 1]{mat80}. Hence it is easy to
 see that $\cX$ satisfies property ($S_2$) since both $\cX_t$ and $T$
 are normal. Hence $\cX$ is normal.
 
(3) Let $\cZ_{T'}:=\cZ\times_T T'$, then $\cX_{T'}\setminus\cZ_{T'}$
is smooth over $T'$. Since the fibers of $\pi_{T'}$ and $T'$ are
irreducible, we know that $\cX_{T'}$ is also irreducible. Thus
the same argument of (2) implies that $\cX_{T'}$ satisfies 
both ($R_1$) and ($S_2$), which means $\cX_{T'}$ is normal. Since
$\pi|_{\cX\setminus\cZ}$ is smooth, we know that $K_{\cX_{T'}/T'}+\cD_{T'}$
and $g^*(K_{\cX/T}+\cD)$ are $\bQ$-linearly equivalent after restriting
to $\cX_{T'}\setminus \cZ_{T'}$. Since $\cZ_{T'}$ is of 
codimension at least $2$ in $\cX_{T'}$, the $\bQ$-linear equivalence over
$\cX_{T'}\setminus\cZ_{T'}$ extends to $\cX_{T'}$. Thus we finish
the proof.
\end{proof}

\begin{defn}
\begin{enumerate}[label=(\alph*)]
 \item Let $Y$ be a normal projective variety. Let $E$
 be an effective $\bQ$-divisor on $Y$. We say that
 $(Y,E)$ is a \emph{log Fano pair} if $(Y,E)$ is a klt pair and 
 $-(K_Y+E)$ is $\bQ$-Cartier and ample. We say $Y$ is a 
 \emph{$\bQ$-Fano variety} if $(Y,0)$ is a log Fano pair.
 \item
 Let $T$ be a normal variety. A $\bQ$-Gorenstein  family
 of klt pairs $\varphi:(\cY,\cE)\to T$ is called a 
 \emph{$\bQ$-Gorenstein flat family of log Fano pairs} if $\varphi$ is
 proper and $-(K_{\cY/T}+\cE)$ is $\varphi$-ample.
\end{enumerate}

\end{defn}

The following proposition is a natural generalization of \cite[Proposition A.2 and A.3]{blu16}. It should be well-known to experts. The proof is omitted because it is essentially the same as \cite[Appendix A]{blu16}.

\begin{prop}\label{lctsemicont}
Let $\pi:(\cX,\cD)\to T$ be a $\bQ$-Gorenstein flat family of klt pairs over a normal variety $T$. Let $\fa$ be an ideal sheaf of $\cX$. Then
\begin{enumerate}
\item The function $t\mapsto \lct(\cX_t,\cD_t;\fa_t)$ on $T$ is constructible;
\item If in addition $V(\fa)$ is proper over $T$, then the function $t\mapsto \lct(\cX_t,\cD_t;\fa_t)$ is lower semi-continuous with respect to the Zariski topology on $T$.
\end{enumerate}
\end{prop}

\section{Comparison of normalized volumes and normalized colengths}
\subsection{Normalized colengths of klt singularities}\label{hatl}
\begin{defn} Let  $x\in (X,D)$ be a klt singularity over an algebraically closed field $\bk$ of characteristic $0$.
Denote its local ring by $(R,\fm):=(\cO_{x,X},\fm_x)$. 
\begin{enumerate}[label=(\alph*)]
\item Given constants $c\in\bR_{>0}$ and $k\in\bN$, we define
the \emph{normalized colength of $x\in (X,D)$ with respect
to $c,k$} as 
\[
\widehat{\ell_{c,k}}(x,X,D):=n!\cdot\inf_{\substack{\fm^k\subset\fa\subset\fm\\\ell(R/\fa)\geq ck^n}}\lct(\fa)^n\cdot\ell(R/\fa).
\]
Here $\fa$ is a $\fm$-primary ideal.
\item 
Given a constant $c\in\bR_{>0}$, we define the \emph{asymptotic normalized colength function
of $x\in (X,D)$ with respect to $c$} as 
\[
\widehat{\ell_{c,\infty}}(x,X,D):=\liminf_{k\to\infty}\widehat{\ell_{c,k}}(x,X,D).
\]
\end{enumerate}
\end{defn}
It is clear that $\widehat{\ell_{c,k}}$ is an increasing function in $c$. The main result in this section is the following theorem.

\begin{thm}\label{hvolcolength}
For any klt singularity $x\in (X,D)$ over an algebraically closed field $\bk$ of characteristic $0$, there exists
$c_0=c_0(x,X,D)>0$ such that 
\begin{equation}\label{eq_hvolcolength}
 \widehat{\ell_{c,\infty}}(x,X,D)=\hvol(x, X, D) \quad\textrm{ whenever }0< c\leq c_0.
\end{equation}
\end{thm}

\begin{proof}
We first show the ``$\leq$'' direction.
Let us take a sequence of valuations $\{v_i\}_{i\in\bN}$ such that
$\lim_{i\to\infty}\hvol(v_i)=\hvol(x,X,D)$. We may rescale
$v_i$ so that $v_i(\fm)=1$ for any $i$. Since $\{\hvol(v_i)\}_{i\in\bN}$ are
bounded from above, by \cite[Theorem 1.1]{li15a} we know that
there exists $C_1>0$ such that $A_{(X,D)}(v_i)\leq C_1$ for any 
$i\in\bN$. Then by Li's Izumi type inequality \cite[Theorem 3.1]{li15a},
there exists $C_2>0$ such that $\ord_{\fm}(f)\leq v_i(f)\leq 
C_2\ord_{\fm}(f)$ for any $i\in\bN$ and any $f\in R$.
As a result, we have $\fm^{k}\subset\fa_{k}(v_i)\subset\fm^{\lceil
k/C_2\rceil}$ for any $i,k\in \bN$.
Thus $\ell(R/\fa_k(v_i))\geq\ell(R/\fm^{\lceil k/C_2\rceil})\sim
\frac{\e(\fm)}{n!C_2^{n}}k^n$. Let us take $c_0=\frac{e(\fm)}{2n!C_2^n}$,
then for $k\gg 1$ we have $\ell(R/\fa_k(v_i))\geq c_0 k^n$
for any $i\in\bN$. Therefore, for any $i\in\bN$ we have
\[
 \widehat{\ell_{c_0,\infty}}(x,X,D)\leq n!\liminf_{k\to\infty}\lct(\fa_k(v_i))^n\ell(R/\fa_k(v_i))
 = \lct(\fa_\bullet(v_i))^n\vol(v_i)\leq \hvol(v_i).
\]
In the last inequality we use $\lct(\fa_\bullet(v_i))\leq A_{(X,D)}(v_i)$
as in the proof of \cite[Theorem 27]{liu16}.
Thus $ \widehat{\ell_{c_0,\infty}}(x,X,D)\leq\lim_{i\to\infty}\hvol(v_i)=\hvol(x,X,D)$.
This finishes the proof of the ``$\leq$'' direction.

For the ``$\geq$'' direction, we will show that 
$\widehat{\ell_{c,\infty}}(x,X,D)\geq \hvol(x,X,D)$ for
any $c>0$. By a logarithmic version of the
Izumi type estimate \cite[Theorem 3.1]{li15a},
there exists a constant $c_1=c_1(x,X,D)>0$ such that 
$v(f)\leq c_1 A_{(X,D)}(v)\ord_{\fm}(f)$ for any valuation
$v\in \Val_{X,x}$ and any function $f\in R$.
For any $\fm$-primary ideal $\fa$, there exists a
divisorial valuation $v_0\in\Val_{X,x}$ computing $\lct(\fa)$
by \cite[Lemma 26]{liu16}. Hence we have the following Skoda
type estimate:
\begin{align*}
 \lct(\fa)=\frac{A_{(X,D)}(v_0)}{v_0(\fa)}\geq \frac{A_{X,D}(v_0)}
 {c_1 A_{(X,D)}(v_0)\ord_{\fm}(\fa)} =\frac{1}{c_1\ord_{\fm}(\fa)}.
\end{align*}
Let $0<\delta<1$ be a positive number. If
$\fa\not\subset\fm^{\lceil\delta k\rceil}$ and $\ell(R/\fa)\geq c k^n$, then 
\[
\lct(\fa)^n\cdot\ell(R/\fa)\geq \frac{ ck^n}
{c_1^n(\lceil\delta k\rceil-1)^n}\geq \frac{ c}{c_1^n\delta^n}.
\]
If we choose $\delta$ sufficiently small such that
$\delta^n\cdot c_1^n\hvol(x,X,D)\leq n! c$, then for any
$\fm$-primary ideal $\fa$ satisfying $\fm^k\subset\fa\not\subset
\fm^{\lceil\delta k\rceil}$ and $\ell(R/\fa)\geq ck^n$ we have
\[
 n!\cdot \lct(\fa)\cdot\ell(R/\fa)\geq \hvol(x,X,D).
\]
Thus it suffices to show
\[
\hvol(x,X,D)\leq n!\cdot\liminf_{k\to\infty}\inf_{\substack{\fm^k\subset\fa\subset\fm^{\lceil\delta k\rceil}\\\ell(R/\fa)\geq ck^n}}\lct^n(\fa)\ell(R/\fa).
\]
By Lemma \ref{colengthmult}, we know that for any $\epsilon>0$
there exists $k_0=k_0(\delta,\epsilon,(R,\fm))$ such that for any $k\geq k_0$ we have
\[
n!\cdot\inf_{\fm^k\subset\fa\subset\fm^{\lceil\delta k\rceil}
}\lct^n(\fa)\ell(R/\fa)\geq (1-\epsilon)\inf_{\fm^k\subset\fa\subset\fm^{\lceil\delta k\rceil}}
\lct(\fa)^n\e(\fa)\geq (1-\epsilon)\hvol(x,X,D).
\]
Hence the proof is finished.
\end{proof}

The following result on comparison between colengths 
and multiplicities is crucial in the proof of Theorem 
\ref{hvolcolength}. Note that Lemma \ref{colengthmult}
is a special case of Lech's inequality \cite[Theorem 3]{lec64} when $R$ is a regular
local ring.

\begin{lem}\label{colengthmult}
Let $(R,\fm)$ be an $n$-dimensional analytically irreducible
Noetherian local domain. Assume that the residue field
$R/\fm$ is algebraically closed. Then for any positive 
numbers $\delta,\epsilon\in(0,1)$, 
there exists $k_0=k_0(\delta,\epsilon,(R,\fm))$ such that 
for any $k\geq k_0$ and any ideal 
$\fm^{k}\subset\fa\subset\fm^{\lceil\delta k\rceil}$, we have
\[
n!\cdot\ell(R/\fa)\geq (1-\epsilon)\e(\fa).
\]
\end{lem}

\begin{proof}
By \cite[7.8]{kk14} and \cite[Section 4]{cut13}, $R$ admits a 
\emph{good} valuation $\nu:R\to\bZ^n$ for some total order on $\bZ^{n}$.
 Let $\cS:=\nu(R\setminus\{0\})\subset\bN^n$ and $C(\cS)$ be the closed convex hull of $\cS$. Then we know that
\begin{itemize}
\item $C(\cS)$ is a strongly convex cone;
\item There exists a linear functional $\xi:\bR^n\to\bR$ such that $C(\cS)\setminus\{0\}\subset\xi_{>0}$;
\item There exists $r_0\geq 1$ such that for any $f\in R\setminus\{0\}$, we have
\begin{equation}\label{eq_izumi}
\ord_{\fm}(f)\leq \xi(\nu(f))\leq r_0\ord_{\fm}(f).
\end{equation}
\end{itemize}

Suppose $\fa$ is an ideal satisfying $\fm^k\subset\fa\subset\fm^{\lceil\delta k\rceil}$. Then we have $\nu(\fm^k)\subset\nu(\fa)\subset\nu(\fm^{\lceil\delta k\rceil})$. By \eqref{eq_izumi}, we know that 
\[
\cS\cap\xi_{\geq r_0 k}\subset 
\nu(\fa)\subset\cS\cap\xi_{\geq \delta k}.
\]
Similarly, we have $\cS\cap\xi_{\geq r_0 ik}\subset 
\nu(\fa^i)\subset\cS\cap\xi_{\geq \delta ik}$ for any positive integer $i$.

Let us define a semigroup $\Gamma\subset\bN^{n+1}$ as follows:
\[
\Gamma:=\{(\alpha,m)\in\bN^n\times\bN\colon x\in\cS\cap\xi_{\leq 2r_0 m}\}.
\]
For any $m\in\bN$, denote by $\Gamma_m:=\{\alpha\in\bN^n\colon(\alpha,m)\in\Gamma\}$. It is easy to see $\Gamma$ satisfies \cite[(2.3-5)]{lm09}, thus \cite[Proposition 2.1]{lm09} implies
\[
\lim_{m\to\infty}\frac{\#\Gamma_m}{m^n}=\vol(\Delta),
\]
where $\Delta:=\Delta(\Gamma)$ is a convex body in $\bR^{n}$ 
defined in \cite[Section 2.1]{lm09}. It is easy to see that $\Delta= C(\cS)\cap\xi_{\leq 2r_0}$.

Let us define $\Gamma^{(k)}:=\{(\alpha,i)\in\bN^n\times\bN\colon(\alpha,ik)\in\Gamma\}$. Then we know that $\Delta^{(k)}:=\Delta(\Gamma^{(k)})=k\Delta$. For an ideal $\fa$ and $k\in\bN$ satisfying $\fm^k\subset\fa\subset\fm^{\lceil\delta k\rceil}$, we define 
\[
\Gamma_{\fa}^{(k)}:=\{(\alpha,i)\in\Gamma^{(k)}\colon\alpha\in\nu(\fa^i)\}.
\]
Then it is clear that $\Gamma_{\fa}^{(k)}$ also satisfies 
\cite[(2.3-5)]{lm09}. Since  $\nu(\fa^i)=(\cS\cap\xi_{> 2r_0ik})
\cup\Gamma_{\fa,i}^{(k)}$ and $R/\fm$ is algebraically 
closed, we have $\ell(R/\fa^i)=\#(\Gamma_i^{(k)}\setminus\Gamma_{\fa,i}^{(k)})$
because $\nu$ has one-dimensional leaves. Again by \cite[Proposition 2.11]{lm09}, we have
\[
n!\e(\fa)=\lim_{i\to\infty}\frac{\ell(R/\fa^i)}{i^n}=\lim_{i\to\infty}\frac{\#(\Gamma_i^{(k)}\setminus\Gamma_{\fa,i}^{(k)})}{i^n}
=\vol(\Delta^{(k)})-\vol(\Delta_{\fa}^{(k)}),
\]
where $\Delta_{\fa}^{(k)}:=\Delta(\Gamma_{\fa}^{(k)})$. Since $\Gamma_{\fa,i}^{(k)}\subset\nu(\fa^i)\subset\xi_{\geq\delta ik}$, we know that $\Delta_{\fa}^{(k)}\subset\xi_{\geq\delta k}$. Denote by $\Delta':=C(\cS)\cap\xi_{< \delta}$, then it is clear that $\Delta_{\fa}^{(k)}\subset k(\Delta\setminus\Delta')$.

On the other hand, 
\[
\ell(R/\fa)=\#(\Gamma_{1}^{(k)}\setminus\Gamma_{\fa,1}^{(k)})
\geq\#\Gamma_{k}-\#(\Delta_{\fa}^{(k)}\cap\bZ^n).
\]
Denote by $\Delta_{\fa,k}:=\frac{1}{k}\Delta_{\fa}^{(k)}$,
then $\Delta_{\fa,k}\subset\Delta\setminus\Delta'$. Since $\vol(\Delta_{\fa,k})\leq \vol(\Delta)-\vol(\Delta')$, there exists positive numbers $\epsilon_1,\epsilon_2$ depending only on $\Delta$ and $\Delta'$ such that 
\begin{equation}\label{ineq1}
\vol(\Delta_{\fa,k})\leq\vol(\Delta)-\vol(\Delta')\leq\left(1-\frac{\epsilon_1}{\epsilon}\right)\vol(\Delta) -\frac{\epsilon_2}{\epsilon}.
\end{equation}
Let us pick $k_0$ such that for any $k\geq k_0$ and any $\fm^k\subset\fa\subset\fm^{\lceil\delta k\rceil}$, we have
\[
\frac{\#\Gamma_k}{k^n}\geq (1-\epsilon_1)\vol(\Delta),\qquad
\frac{\#(\Delta_{\fa}^{(k)}\cap\bZ^n)}{k^n}\leq\vol(\Delta_{\fa,k})+\epsilon_2.
\]
Here the second inequality is guaranteed by applying Proposition \ref{convexgeo} to $\Delta_{\fa,k}$
as a sub convex body of a fixed convex body $\Delta$. Thus
\begin{align*}
    \frac{\ell(R/\fa)-(1-\epsilon)n!\e(\fa)}{k^n}&
    \geq \frac{\#\Gamma_k}{k^n}-\frac{\#(\Delta_{\fa}^{(k)}\cap\bZ^n)}{k^n}-(1-\epsilon)(\vol(\Delta)-\vol(\Delta_{\fa,k}))\\
   & \geq (1-\epsilon_1)\vol(\Delta)-\vol(\Delta_{\fa,k})-\epsilon_2-(1-\epsilon)(\vol(\Delta)-\vol(\Delta_{\fa,k}))\\
   & =(\epsilon-\epsilon_1)\vol(\Delta)-\epsilon(\Delta_{\fa,k})-\epsilon_2\\
   &\geq 0.
\end{align*}
Here the last inequality follows from \eqref{ineq1}. Hence we finish the proof.
\end{proof}

\subsection{Normalized volumes under field extensions}\label{sec_fieldext}
In the rest of this section, we use Hilbert schemes to describe
normalized volumes of singularities after a field extension $\bK/\bk$.
Let $(X,D)$ be a klt pair over $\bk$. Let $x\in X$ be a 
$\bk$-rational point. Let $Z_k:=\Spec(\cO_{x,X}/\fm_{x,X}^k)$ 
be the $k$-th thickening of $x$.  Consider the Hilbert scheme
$H_{k,d}:=\mathrm{Hilb}_{d}(Z_k/\bk)$. For any field extension
$\bK/\bk$ we know that $H_{k,d}(\bK)$ parametrizes ideal sheaves
$\fc$ of $X_\bK$ satisfying $\fc\supset\fm_{x_{\bK},X_{\bK}}^k$
and $\ell(\cO_{x_{\bK},X_{\bK}}/\fc)=d$. In particular,
any scheme-theoretic point $h\in H_{k,d}$ corresponds to an ideal $\fb$ of $\cO_{x_{\kappa(h)},X_{\kappa(h)}}$ satisfying those two conditions, and we denote by $h=[\fb]$.

\begin{prop}\label{fieldext}
Let $\bk$ be a field of characteristic $0$. Let $(X,D)$ be a klt pair over $\bk$. Let $x\in X$ be a $\bk$-rational point. Then 
\begin{enumerate}
    \item For any field extension $\bK/\bk$ with $\bK$ algebraically closed, we have 
    \[
    \widehat{\ell_{c,k}}(x_{\bK}, X_{\bK},D_{\bK})=
    n!\cdot\inf_{d\geq ck^n, ~ [\fb]\in H_{k,d}}
    d\cdot\lct(X_{\kappa([\fb])}, D_{\kappa([\fb])};\fb)^n.
    \]
    \item With the assumption of (1), we have $$\hvol(x_{\bK}, X_{\bK},D_{\bK})=\hvol(x_{\bar{\bk}}, X_{\bar{\bk}},D_{\bar{\bk}}).$$
\end{enumerate}
\end{prop}

\begin{proof} (1) We first prove the ``$\geq$'' direction.
By definition, $\widehat{\ell}_{c,k}(x_{\bK}, X_{\bK}, D_{\bK})$ is the 
infimum of $n!\cdot\lct(X_{\bK},D_{\bK};\fc)^n\ell(\cO_{X_{\bK}}/\fc)$ where
$\fc$ is an ideal on $X_{\bK}$ satisfying $\fm_{x_{\bK}}^k\subset
\fc\subset\fm_{x_{\bK}}$ and $\ell(\cO_{X_{\bK}}/\fc)=:d\geq ck^n$. Hence $[\fc]$ represents a point in 
$H_{k,d}(\bK)$. Suppose $[\fc]$ is lying over a scheme-theoretic point $[\fb]\in H_{k,d}$, then it is clear that $(X_{\bK},D_{\bK},\fc)\cong (X_{\kappa([\fb])},D_{\kappa([\fb])},\fb)\times_{\Spec(\kappa([\fb]))}\Spec(\bK)$. Hence $\lct(X_{\bK},D_{\bK};\fc)=\lct(X_{\kappa([\fb])},D_{\kappa([\fb])};\fb)$, and the ``$\geq$'' direction is proved.

Next we prove the ``$\leq$'' direction. By Proposition \ref{lctsemicont}, we know that the function $[\fb]\mapsto\lct(X_{\kappa([\fb])},D_{\kappa([\fb])}; \fb)$ on $H_{k,d}$ is constructible and lower semi-continuous.
Denote by $H_{k,d}^{\mathrm{cl}}$ the set of closed points in $H_{k,d}$. Since the set of closed points are dense in any stratum of $H_{k,d}$ with respect to the $\lct$ function, we have the following equality:
\[
n!\cdot\inf_{d\geq ck^n, ~ [\fb]\in H_{k,d}}
    d\cdot\lct(X_{\kappa([\fb])}, D_{\kappa([\fb])};\fb)^n
    =n!\cdot\inf_{d\geq ck^n, ~ [\fb]\in H_{k,d}^{\mathrm{cl}}}
    d\cdot\lct(X_{\kappa([\fb])}, D_{\kappa([\fb])};\fb)^n
\]
Any $[\fb]\in H_{k,d}^{\mathrm{cl}}$ satisfies that $\kappa([\fb])$ is an algebraic extension of $\bk$. Since $\bK$ is algebraically closed, $\kappa([\fb])$ can be embedded into $\bK$ as a subfield. Hence there exists a point $[\fc]\in H_{k,d}(\bK)$ lying over $[\fb]$. Thus similar arguments implies that $\lct(X_{\bK},D_{\bK};\fc)=\lct(X_{\kappa([\fb])},D_{\kappa([\fb])};\fb)$, and the ``$\leq$'' direction is proved.
\medskip

(2) From (1) we know that $\widehat{\ell_{c,k}}(x_{\bK},X_{\bK},D_{\bK})=\widehat{\ell_{c,k}}(x_{\bar{\bk}},X_{\bar{\bk}},D_{\bar{\bk}})$ for any $c,k$. Hence it follows from Theorem \ref{hvolcolength}.
\end{proof}

The following corollary is well-known to experts.
We present a proof here using normalized volumes.

\begin{cor}
 Let $(Y,E)$ be a log Fano pair over
 a field $\bk$ of characteristic 
 $0$. The following are equivalent:
 \begin{enumerate}[label=(\roman*)]
  \item $(Y_{\bar{\bk}},E_{\bar{\bk}})$ is log K-semistable;
  \item $(Y_{\bK},E_{\bK})$ is log K-semistable
  for some field extension $\bK/\bk$ with $\bK=\overline{\bK}$;
  \item $(Y_{\bK},E_{\bK})$ is log K-semistable
  for any field extension $\bK/\bk$ with $\bK=\overline{\bK}$.
 \end{enumerate}
 We say that $(Y,E)$ is \emph{geometrically log K-semistable} if 
 one (or all) of these conditions holds.
\end{cor}

\begin{proof}
Let us take the affine cone $X=C(Y,L)$ with $L=-r(K_Y+E)$  Cartier. Let $D$ be the $\bQ$-divisor on $X$ corresponding to $E$. Denote by$x\in X$ the cone vertex of $X$. Let $\bK/\bk$ be a field extension with $\bK=\overline{\bK}$. 
Then  Theorem \ref{ksscone} implies that $(Y_{\bK},E_{\bK})$ is log K-semistable if and only if $\hvol(x_{\bK}, X_{\bK},D_{\bK})=r^{-1}(-K_Y-E)^{n-1}$. Hence the corollary is a consequence of Proposition \ref{fieldext} (2).
\end{proof}

We finish this section with a natural speculation. Suppose 
$x\in (X,D)$ is a klt singularity over a field $\bk$ of characteristic
zero that is not necessarily algebraically closed. The definition of
normalized volume of singularities extend verbatimly to
$x\in (X,D)$ which we also denote by $\hvol(x,X,D)$. Then
we expect $\hvol(x,X,D)=\hvol(x_{\bar{\bk}},X_{\bar{\bk}}, D_{\bar{\bk}})$,
i.e. normalized volumes are stable under base change to algebraic
closures. Such a speculation should be a consequence of the \emph{Stable
Degeneration Conjecture (SDC)} stated in \cite[Conjecture 7.1]{li15a}
and \cite[Conjecture 1.2]{lx17} which roughly says that a $\hvol$-minimizing
valuation $v_{\min}$ over $x_{\bar{\bk}}\in (X_{\bar{\bk}}, D_{\bar{\bk}})$ is unique and quasi-monomial,
so $v_{\min}$ is invariant under the action of $\mathrm{Gal}(\bar{\bk}/\bk)$
and hence has the same normalized volume as its restriction
to $x\in (X,D)$.

\section{Proofs and applications}

\subsection{Proofs}\label{sec_proofs}
The following theorem is a stronger result that implies Theorem \ref{mainthm}.

\begin{thm}\label{weaksc}
Let $\pi:(\cX,\cD)\to T$ together with a section $\sigma: T\to \cX$ be a $\bQ$-Gorenstein flat family of klt singularities over a field $\bk$ of characteristic $0$.
Then for any point $o\in T$, there exists an intersection
$U$ of countably many Zariski open neighborhoods of $o$, 
such that $\hvol(\sigma(\bar{t}),\cX_{\bar{t}},\cD_{\bar{t}})\geq\hvol(\sigma(\bar{o}),\cX_{\bar{o}},\cD_{\bar{o}})$
for any point $t\in U$. In particular, if $t$ is a generalization
of $o$ then $\hvol(\sigma(\bar{t}),\cX_{\bar{t}},\cD_{\bar{t}})\geq\hvol(\sigma(\bar{o}),\cX_{\bar{o}},\cD_{\bar{o}})$.
\end{thm}

\begin{proof}
Let $\cZ_k\to T$ be the $k$-th thickening of the section $\sigma$,
i.e. $\cZ_{k,t}=\Spec(\cO_{\cX_t}/\fm_{\sigma(t),\cX_t}^k)$. 
Let $d_k:=\max_{t\in T}\ell(\cO_{\sigma(t),\cX_t}/\fm_{\sigma(t),\cX_t}^k)$.
For any $d\in\bN$,
denote $\cH_{k,d}:=\mathrm{Hilb}_{d}(\cZ_k/T)$.
Since $\cZ_k$ is proper over $T$, we know that
$\cH_{k,d}$ is also proper over $T$. Let $\cH_{k,d}^{\mathrm{n}}$ be the normalization of $\cH_{k,d} $. Denote by $\tau_{k,d}:\cH_{k,d}\to T$. After pulling back the 
universal ideal sheaf on $\cX\times_T\cH_{k,d}$ over $\cH_{k,d}$ to $\cH_{k,d}^{\mathrm{n}}$,
we obtain an ideal sheaf $\fb_{k,d}$ on $\cX\times_T\cH_{k,d}^{\mathrm{n}}$.
Denote by $\pi_{k,d}: (\cX\times_T\cH_{k,d}^{\mathrm{n}}, \cD\times_T \cH_{k,d}^{\mathrm{n}})\to \cH_{k,d}^{\mathrm{n}}$
the projection, then $\pi_{k,d}$ provides a $\bQ$-Gorenstein flat family of klt pairs.

Following the notation of Proposition \ref{fieldext}, assume $h$ is 
scheme-theoretic point of $\cH_{k,d}^{\mathrm{n}}$ lying over 
$[\fb]\in\cH_{k,d}$. Denote by $t=\tau_{k,d}([\fb])\in T$. By 
construction, the ideal sheaf $\fb_{k,d,h}$ on $\cX\times_{T}
\Spec(\kappa(h))$ is the pull back of  $\fb$ under the flat base 
change $\Spec(\kappa(h))\to\Spec(\kappa([\fb]))$. Hence 
$$\lct((\cX,\cD)\times_{T}\Spec(\kappa(h));\fb_{k,d,h})=
\lct((\cX,\cD)\times_{T}\Spec(\kappa([\fb]));\fb).$$
For simplicity, we abbreviate the above equation to $\lct(\fb_{k,d,h})=\lct(\fb)$.
Applying Proposition \ref{lctsemicont} to the family $\pi_{k,d}$ and the ideal $\fb_{k,d}$ implies that the function $\Phi^{\mathrm{n}}:\cH_{k,d}^{\mathrm{n}}\to\bR_{>0}$ defined as
$\Phi^{\mathrm{n}}(h):=\lct(\fb_{k,d,h})$ is constructible and lower semi-continuous with respect to the Zariski topology on $\cH_{k,d}^{\mathrm{n}}$. Since $\lct(\fb_{k,d,h})=\lct(\fb)$, $\Phi^{\mathrm{n}}$ descend to a function $\Phi$ on $\cH_{k,d}$ as $\Phi([\fb]):=\lct(\fb)$. Since $\cH_{k,d}$ is proper over $T$, we know that the function $\phi:T\to \bR_{>0}$ defined as
\[
\phi( t):=n!\cdot\min_{\substack{ck^n\leq d\leq d_k\\ [\fb]\in\tau_{k,d}^{-1}(t)}}\Phi([\fb])^n
\]
is constructible and lower semi-continuous with respect to the Zariski topology on $T$. Then Proposition \ref{fieldext} implies $\phi(t)=\widehat{\ell_{c,k}}(\sigma(\bar{t}), \cX_{\bar{t}},D_{\bar{t}})$. Thus we conclude that $t\mapsto\widehat{\ell_{c,k}}(\sigma(\bar{t}), \cX_{\bar{t}},D_{\bar{t}})$ is constructible and lower semi-continuous with respect to the Zariski topology on $T$. Hence for any $k,m\in\bN$ there exists a Zariski open neighborhood $U_{k,m}$ of $o$, such that 
\[
\widehat{\ell_{1/m,k}}(\sigma(\bar{t}), \cX_{\bar{t}},D_{\bar{t}})\geq \widehat{\ell_{1/m,k}}(\sigma(\bar{o}), \cX_{\bar{o}},D_{\bar{o}})
\quad\textrm{ whenever }t\in U_{k,m}. 
\]
Let $U:=\cap_{k,m} U_{k,m}$, then for any $m\in\bN$ and any $t\in U$ we have
$\widehat{\ell_{1/m,\infty}}(\sigma(\bar{t}), \cX_{\bar{t}},D_{\bar{t}})\geq \widehat{\ell_{1/m,\infty}}(\sigma(\bar{o}), \cX_{\bar{o}},D_{\bar{o}})$. By Theorem \ref{hvolcolength}, for any $t\in U$ we have
\begin{align*}
\hvol(\sigma(\bar{t}),\cX_{\bar{t}},D_{\bar{t}}) & =\lim_{m\to\infty}\widehat{\ell_{1/m,\infty}}(\sigma(\bar{t}), \cX_{\bar{t}},D_{\bar{t}})\\&\geq \lim_{m\to\infty}\widehat{\ell_{1/m,\infty}}(\sigma(\bar{o}), \cX_{\bar{o}},D_{\bar{o}})\\ &=\hvol(\sigma(\bar{o}), \cX_{\bar{o}},D_{\bar{o}}).
\end{align*}
The proof is finished.
\end{proof}

The following theorem is a stronger result that implies
Theorem \ref{openk}.

\begin{thm}\label{openkgen}
Let $\varphi:(\cY,\cE)\to T$ be a $\bQ$-Gorenstein flat 
family of log Fano pairs over a field $\bk$ of
characteristic $0$. Assume that some geometric fiber 
$(\cY_{\bar{o}},\cE_{\bar{o}})$ is log K-semistable 
for a point $o\in T$. Then
\begin{enumerate}
\item There exists an intersection
$U$ of countably many Zariski open neighborhoods of $o$,
such that $(\cY_{\bar{t}},\cE_{\bar{t}})$ is log
K-semistable for any  point $t\in T$. If, in addition,
$\bk=\bar{\bk}$ is uncountable, then $(\cY_{t},\cE_{t})$ 
is log K-semistable for a very general closed point $t\in T$.
\item The geometrically log K-semistable locus $$T^{\textrm{K-ss}}:=\{t\in T\colon (\cY_{\bar{t}},\cE_{\bar{t}})\textrm{ is log K-semistable}\}$$ is stable under generalization.
\end{enumerate}
\end{thm}

\begin{proof}
(1) For $r\in\bN$ satisfying $\cL=-r(K_{\cY/T}+\cE)$
is Cartier, we define the \emph{relative affine cone} $\cX$ of $(\cY,\cL)$ by
\[
 \cX:=\Spec_{T}\oplus_{m\geq 0}\varphi_*(\cL^{\otimes m}).
\]
Assume $r$ is sufficiently large, then it is easy to see that $\varphi_*(\cL^{\otimes m})$
is locally free on $T$ for all $m\in\bN$. Thus we have 
$\cX_t\cong \Spec\oplus_{m\geq 0}H^0(\cY_t,\cL_t^{\otimes m}):=C(\cY_t,\cL_t)$.
Let $\cD$ be the $\bQ$-divisor on $\cX$ corresponding to $\cE$. 
By \cite[Section 3.1]{kol13}, the projection $\pi:(\cX,\cD)\to T$ together
with the section of cone vertices $\sigma:T\to\cX$ is
a $\bQ$-Gorenstein flat family of klt singularities.

Since $(\cY_{\bar{o}},\cE_{\bar{o}})$ is K-semistable, Theorem \ref{ksscone} implies
$$\hvol(\sigma(\bar{o}),\cX_{\bar{o}}, \cD_{\bar{o}})=r^{-1}(-K_{\cY_{\bar{o}}}-\cE_{\bar{o}})^{n-1}.$$
Then by Theorem \ref{weaksc}, there exists an intersection $U$ of countably many Zariski open neighborhoods of $o$, such that $\hvol(\sigma(\bar{t}),\cX_{\bar{t}}, \cD_{\bar{t}})\geq \hvol(\sigma(\bar{o}),\cX_{\bar{o}}, \cD_{\bar{o}})$ for any $t\in U$. Since the global volumes of log Fano pairs are constant in $\bQ$-Gorenstein flat families, we have
\[
\hvol(\sigma(\bar{t}),\cX_{\bar{t}}, \cD_{\bar{t}})\geq \hvol(\sigma(\bar{o}),\cX_{\bar{o}}, \cD_{\bar{o}})=
r^{-1}(-K_{\cY_{\bar{o}}}-\cE_{\bar{o}})^{n-1}=r^{-1}(-K_{\cY_{\bar{t}}}-\cE_{\bar{t}})^{n-1}.
\]
Then Theorem \ref{ksscone} implies that $(\cY_{\bar{t}},\cE_{\bar{t}})$ is K-semistable for any $t\in U$.
\medskip

(2) Let $o\in T^{\textrm{K-ss}}$ be a scheme-theoretic point. Then by Theorem \ref{openkgen} there exists countably many Zariski open neighborhoods $U_i$ of $o$ such that $\cap_i U_i\subset T^{\textrm{K-ss}}$.
If $t$ is a generalization of $o$, then $t$ belongs to all Zariski open neighborhoods of $o$, so $t\in T^{\textrm{K-ss}}$.
\end{proof}

\begin{proof}[Proof of Theorem \ref{openk}]
It is clear that (1) and (2) follows from Theorem \ref{openkgen}.
For (3), we only need to replace Theorem \ref{weaksc} by Conjecture 
\ref{mainconj}, then the same argument in the proof of Theorem \ref{openkgen} (1)
works.
\end{proof}

The following corollary is a stronger result that implies
Corollary \ref{specialdeg}.

\begin{cor}
 Let $\pi:(\cY,\cE)\to T$ be a $\bQ$-Gorenstein family of complex log Fano pairs. Assume that $\pi$ is isotrivial over a Zariski open subset $U\subset T$, and $(\cY_o,\cE_o)$ is log K-semistable for a closed point $o\in T\setminus U$. Then $(\cY_t,\cE_t)$ is log K-semistable for any $t\in U$.
\end{cor}

\begin{proof}
Since $(\cY_o,\cE_o)$ is log K-semistable, Theorem \ref{openkgen} implies that $(\cY_t,\cE_t)$ is log K-semistable for very general closed point $t\in T$. Hence there exists (hence any) $t\in U$ such that $(\cY_t,\cE_t)$ is log K-semistable.
\end{proof}

\subsection{Applications}\label{sec_appl}
In this section we present applications of Theorem \ref{mainthm}. The
following theorem generalizes the inequality part of \cite[Theorem A.4]{liux17}.
\begin{thm}\label{maxhvol}
Let $x\in (X,D)$ be a complex klt singularity of dimension $n$. Let
$a$ be the largest coefficient of components of $D$ containing $x$. Then
$\hvol(x,X,D)\leq (1-a)n^n$.
\end{thm}

\begin{proof}
Suppose $D_i$ is the component of $D$ containing $x$ with
coefficient $D$. Let $D_i^{\mathrm{n}}$ be the normalization
of $D_i$. By applying Theorem \ref{mainthm} to $\mathrm{pr}_2:
(X\times D_i^{\mathrm{n}}, D\times D_i^{\mathrm{n}})\to 
D_i^{\mathrm{n}}$ together with the natural diagonal section
$\sigma:D_i^{\mathrm{n}}\to X\times D_i^{\mathrm{n}}$, we have that
$\hvol(x,X,D)\leq\hvol(y, X, D)$ for a very general closed point 
$y\in D_i$. We may pick $y$ to be a smooth point in 
both $X$ and $D$, then $\hvol(x,X,D)\leq \hvol(0,\bA^n,a \bA^{n-1})$ 
where $\bA^{n-1}$ is a coordinate hyperplane of $\bA^n$. Let us take
local coordinates $(z_1,\cdots,z_n)$ of $\bA^n$ such that $\bA^{n-1}=V(z_1)$.
Then the monomial valuation $v_a$ on $\bA^n$ with weights $((1-a)^{-1},1,\cdots,1)$
satisfies $A_{\bA^n}(v)=\frac{1}{1-a}+(n-1)$, $\ord_{v_a}(\bA^{n-1})=\frac{1}{1-a}$
and $\vol(v_a)=(1-a)$. Hence
\[
 \hvol(x,X,D)\leq \hvol_{0,(\bA^n,a\bA^{n-1})}(v_a)=(A_{\bA^n}(v)-a\ord_{v_a}(\bA^{n-1}))^n
 \cdot\vol(v_a)=(1-a)n^n.
\]
The proof is finished.
\end{proof}

\begin{thm}
Let $(X,D)$ be a klt pair over $\bC$. Let $Z$ be an irreducible subvariety of $X$. Then for a very general closed point $z\in Z$ we have
\[
\hvol(z,X,D)=\sup_{x\in Z}\hvol(x, X,D).
\]
In particular, there exists a countable intersection $U$ of non-empty Zariski open subsets of $Z$ such that $\hvol(\cdot, X,D)|_U$ is constant. \end{thm}

\begin{proof}
Denote by $Z^{\mathrm{n}}$ the normalization of $Z$. Then the proof follows quickly by applying Theorem \ref{mainthm} to $\mathrm{pr}_2:(X\times Z^{\mathrm{n}},D\times Z^{\mathrm{n}})\to Z^{\mathrm{n}}$ together with the natural diagonal section $\sigma:Z^{\mathrm{n}}\to X\times Z^{\mathrm{n}}$.
\end{proof}

Next we study the case when $X$ is a Gromov-Hausdorff limit
of K\"ahler-Einstein Fano manifolds. Note that the function $x\mapsto\hvol(x,X)=n^n\cdot\Theta(x,X)$ is lower semi-continuous with respect to the Euclidean topology on $X$ by \cite{ss17, lx17}. This result together with Theorem \ref{mainthm} provide strong evidence of the special case of  Conjecture \ref{mainconj} on the constructibility and lower semi-continuity of the function $x\mapsto\hvol(x,X)$ for a klt variety $X$.

The following theorem partially generalizes \cite[Lemma 3.3 and Proposition 3.10]{ss17}.
\begin{thm}\label{ghlimit}
Let $X$ be a Gromov-Hausdorff limit of K\"ahler-Einstein Fano manifolds. Let $x\in X$ be any closed point. Then for any finite quasi-\'etale morphism of singularities $\pi:(y\in Y)\to(x\in X)$, we have 
$\deg(\pi)\leq \Theta(x,X)^{-1}$. In particular, we have
\begin{enumerate}
    \item $|\hat{\pi}_1^{\mathrm{loc}}(X,x)|\leq
    \Theta(x,X)^{-1}$.
    \item For any $\bQ$-Cartier Weil divisor $L$ on $X$, we have
    $\mathrm{ind}(x,L)\leq\Theta(x,X)^{-1}$ where $\mathrm{ind}(x,L)$ denotes the
    Cartier index of $L$ at $x$.
\end{enumerate}
\end{thm}
\begin{proof}
 By \cite[Theorem 1.7]{lx17}, the finite degree formula holds for $\pi$, i.e.
 $\hvol(y,Y)=\deg(\pi)\cdot\hvol(x,X)$. Since $\hvol(y,Y)\leq n^n$ by 
 \cite[Theorem A.4]{liux17} or Theorem \ref{maxhvol} and $\hvol(x,X)
 =\Theta(x,X)$ by \cite[Corollary 5.7]{lx17}, we have
 $\deg(\pi)\leq n^n/\hvol(x,X)=\Theta(x,X)^{-1}$.
\end{proof}

\begin{rem}\label{r_localpi1}
If the finite degree formula \cite[Conjecture 4.1]{liux17} were
true for any klt singularity, then clearly $\deg(\pi)\leq n^n/\hvol(x,X)$ holds for any 
finite quasi-\'etale morphism $\pi: (y,Y)\to (x,X)$ between $n$-dimensional klt singularities. 
In particular, we would get an effective upper bound 
$|\hat{\pi}_1^{\mathrm{loc}}(X,x)|\leq n^n/\hvol(x,X)$ where
$\hat{\pi}_1^{\mathrm{loc}}(X,x)$ is known to be finite
by \cite{xu14, bgo17} (see \cite[Theorem 1.5]{liux17} for a
partial result in dimension $3$).
\end{rem}

\begin{thm}
Let $V$ be a K-semistable complex $\bQ$-Fano variety of dimension $(n-1)$. Let $q$ be the largest integer such that there exists a Weil divisor $L$ satisfying $-K_V\sim_{\bQ}qL$. Then 
\[
q\cdot(-K_V)^{n-1}\leq n^n.
\]
\end{thm}

\begin{proof}
 Consider the orbifold cone $X:=C(V,L)=\Spec(\oplus_{m\geq 0}H^0(V,
 \cO_V(\lfloor m L\rfloor))$ with the cone vertex $x\in X$.
 Let $\tilde{X}:=\Spec_V\oplus_{m\geq 0}\cO_V(\lfloor mL\rfloor)$ 
 be the partial resolution of $X$ with exceptional divisor $V_0$.
 Then by \cite[40-42]{kol04}, $x\in X$ is a klt singularity, and $(V_0,0)\cong (V,0)$
 is a K-semistable Koll\'ar component over $x\in X$. Hence \cite[Theorem A]{lx16} implies
 that $\ord_{V_0}$ minimizes $\hvol_{x,X}$. By \cite[40-42]{kol04}
 we have $A_X(\ord_{V_0})=q$, $\vol(\ord_{V_0})=(L^{n-1})$. Hence
 $$\hvol(x,X)=A_X(\ord_{V_0})^n\vol(\ord_{V_0})=q^n(L^{n-1})=q(-K_V)^{n-1},$$
 and the proof is finished since $\hvol(x,X)\leq n^n$ by \cite[Theorem A.4]{liux17}
 or Theorem \ref{maxhvol}.
 \end{proof}

\appendix
\section{Asymptotic lattice points counting in convex bodies}
\label{app}

In this appendix, we will prove the following proposition.

\begin{prop}\label{convexgeo}
For any positive number $\epsilon$, there exists
$k_0=k_0(\epsilon, n)$ such that for any closed convex body
$\Delta\subset [0,1]^n$ and any integer $k\geq k_0$, we have
\begin{equation}\label{ineq_convexgeo}
\left|\frac{\#(k\Delta\cap\bZ^n)}{k^n}-\vol(\Delta)\right|\leq\epsilon.
\end{equation}
\end{prop}

\begin{proof}
We do induction on dimensions. If $n=1$, then $k\Delta$ is a closed
interval of length $k\vol(\Delta)$, hence we know
\[
 k\vol(\Delta)-1\leq\#(k\Delta\cap\bZ)\leq k\vol(\Delta)+1.
\]
So \eqref{ineq_convexgeo} holds for $k_0=\lceil 1/\epsilon\rceil$.

Next, assume that the proposition is true for dimension $n-1$.
Denote by $(x_1,\cdots,x_n)$ the
coordinates of $\bR^n$. Let $\Delta_t:=\Delta\cap\{x_n=t\}$ be
the sectional convex body in $[0,1]^{n-1}$. Let $[t_{-},t_{+}]$ be
the image of $\Delta$ under the projection onto the last coordinate.
Then we know that $\vol(\Delta)=\int_{t_{-}}^{t_+}\vol(\Delta_t)dt$.
By induction hypothesis, there exists $k_1\in \bN$ such that $$
\vol(\Delta_t)-\frac{\epsilon}{3}\leq \frac{\#(k\Delta_t\cap\bZ^{n-1})}{k^{n-1}}
\leq \vol(\Delta_t)+\frac{\epsilon}{3} \quad \textrm{for any }k\geq k_1.
$$
It is clear that
$$\#(k\Delta\cap\bZ^n)=\sum_{t\in[t_-,t_+]\cap\frac{1}{k}\bZ}\#(k\Delta_t\cap\bZ^{n-1}),$$
so for any $k\geq k_1$ we have
\begin{equation}\label{ineq_cg1}
 \left|\#(k\Delta\cap\bZ^n) -k^{n-1}\cdot \sum_{t\in[t_-,t_+]\cap\frac{1}{k}\bZ}\vol(\Delta_t)\right|
 \leq \frac{\epsilon}{3} k^{n-1}\cdot\#([t_-,t_+]\cap\frac{1}{k}\bZ)\leq \frac{2\epsilon}{3} k^n.
\end{equation}

Next, we know that the function $t\mapsto\vol(\Delta_t)^{1/{(n-1)}}$ is concave on $[t_-,t_+]$ by the Brunn-Minkowski theorem.
In particular, we can find $t_0\in [t_-,t_+]$ such that $g(t):=\vol(\Delta_t)$ reaches 
its maximum at $t=t_0$. Hence $g$ is increasing on $[t_-,t_0]$
and decreasing on $[t_0,t_+]$.
Then applying Proposition \ref{integration} to $g|_{[t_-,t_0]}$
and $g|_{[t_0,t_+]}$ respectively yields
\begin{align*}
 \left|\int_{t_-}^{t_0}\vol(\Delta_t)dt-\frac{1}{k}\sum_{t\in[t_-,t_0]\cap\frac{1}{k}\bZ}\vol(\Delta_t)\right|
\leq \frac{2}{k},\\
\left|\int_{t_0}^{t_+}\vol(\Delta_t)dt-\frac{1}{k}\sum_{t\in[t_0,t_+]\cap\frac{1}{k}\bZ}\vol(\Delta_t)\right|
\leq \frac{2}{k}.
\end{align*}
Since $0\leq \vol(\Delta_{t_0})\leq 1$, we have
\begin{equation}\label{ineq_cg2}
 \left|\int_{t_-}^{t_+}\vol(\Delta_t)dt-\frac{1}{k}\sum_{t\in[t_+,t_-]\cap\frac{1}{k}\bZ}\vol(\Delta_t)\right|
 \leq \frac{5}{k}.
\end{equation}
Therefore, by setting $k_0=\max(k_1,\lceil 15/\epsilon\rceil)$, the 
inequality \eqref{ineq_convexgeo}
follows easily by combining \eqref{ineq_cg1} and \eqref{ineq_cg2}.
\end{proof}

\begin{prop}\label{integration}
 For any monotonic function $g:[a,b]\to [0,1]$ and any $k\in\bN$, we have
 \[
  \left|\int_{a}^b g(s) ds-\frac{1}{k}\sum_{t\in[a,b]\cap\frac{1}{k}\bZ}g(t)\right|\leq\frac{2}{k}.
 \]
\end{prop}

\begin{proof}
 We may assume that $g$ is an increasing function.
 Denote $a_k:=\frac{\lceil ka\rceil}{k}$ and $b_k:=\frac{\lfloor kb\rfloor}{k}$,
 so $[a,b]\cap\frac{1}{k}\bZ=[a_k,b_k]\cap\frac{1}{k}\bZ$.
 Since $\int_{t-1/k}^{t} g(s)ds\leq g(t)/k$ whenever $t\in [a_k+1/k,b_k]$,
 we have 
 \[
  \int_{a_k}^{b_k}g(s)ds \leq \frac{1}{k}\sum_{t\in[a_k+1/k,b_k]\cap\frac{1}{k}\bZ}g(t)
  \leq \frac{1}{k}\sum_{t\in[a,b]\cap\frac{1}{k}\bZ}g(t),
 \]
Similarly, $\int_{t}^{t+1/k} g(s)ds\geq g(t)/k$ for any $t\in [a_k, 
b_k-1/k]$, we have
 \[
  \int_{a_k}^{b_k}g(s)ds\geq\frac{1}{k}\sum_{t\in[a_k,b_k-1/k]
  \cap\frac{1}{k}\bZ}g(t)\geq\frac{1}{k}\sum_{t\in[a,b]
  \cap\frac{1}{k}\bZ}g(t)-\frac{1}{k}.
 \]
 It is clear that $a_k\in[a, a+1/k]$ and $b_k\in[b-1/k, b]$,
 so we have
 \[
  \int_{a_k}^{b_k}g(s)ds\geq \int_a^bg(s)ds-\frac{2}{k},\qquad
  \int_{a_k}^{b_k}g(s)ds\leq \int_a^bg(s)ds.
 \]
 As a result, we have
 \[
 \frac{1}{k}\sum_{t\in[a,b]  \cap\frac{1}{k}\bZ}g(t)-\frac{1}
 {k}\leq \int_{a}^b g(s)ds\leq \frac{1}{k}\sum_{t\in[a,b]
  \cap\frac{1}{k}\bZ}g(t)+\frac{2}{k}
 \]
\end{proof}

\end{document}